\documentclass[letterpaper, 10 pt, journal, twoside]{ieeetran}
\usepackage[utf8]{inputenc}

\usepackage[normalem]{ulem}
\usepackage{cite}
\usepackage{amsmath,amssymb,amsfonts,mathtools}
\usepackage{graphicx}
\usepackage[export]{adjustbox}
\usepackage[dvipsnames]{xcolor}
\usepackage{hyperref}
\usepackage{cancel}
\usepackage{caption}
\usepackage{arydshln}
\usepackage{epsfig}
\usepackage{verbatim} 
\usepackage{tikz}

\usepackage{amsthm}

\newcommand{\half}{{\textstyle\frac{1}{2}}}

\newcommand{\reals}{\mathbb{R}}
\newcommand{\transpose}{^\textrm{T}}
\DeclareMathOperator*{\argmin}{arg\,min}

\DeclareMathOperator{\abssq}{ssq}

\newtheorem{theorem}{Theorem}

\newtheorem{lemma}{Lemma}
\newtheorem{definition}{Definition}
\newtheorem{corollary}{Corollary}

\newtheorem{remark}{Remark}

\def\BibTeX{{\rm B\kern-.05em{\sc i\kern-.025em b}\kern-.08em
    T\kern-.1667em\lower.7ex\hbox{E}\kern-.125emX}}

\title{Guaranteed Safe Spacecraft Docking with Control Barrier Functions}

\author{Joseph Breeden and Dimitra Panagou
\thanks{This work was supported by the National Science Foundation Graduate Research Fellowship Program.}%
\thanks{The authors are with the Department of Aerospace Engineering, University of Michigan, Ann Arbor, MI, USA. Email: \texttt{\{jbreeden,dpanagou\}@umich.edu}%
}
}

\definecolor{subsectioncolor}{cmyk}{0, 0, 0, 1}

\begin{document}

\maketitle
\thispagestyle{empty}

\begin{abstract}

This paper presents a strategy for control of a spacecraft docking with a non-maneuvering target in the presence of safety constraints and bounded disturbances. 
The presence of disturbances prevents convergence to a unique docking state, so in our formulation, docking is defined as occurring within a set constructed using prescribed tolerances.
Safety is ensured via application of Robust Control Barrier Functions to render a designated safe set forward invariant for any allowable disturbance. However, this safety strategy necessarily presumes a worst-case disturbance, and thus restricts 
trajectories to a subset of the safe set when a worst-case disturbance is not present. 
The presented controller accounts for this restriction, and guarantees that the spacecraft both remains safe and achieves docking in finite time for any allowable disturbance. The controller is then validated in simulation for a spacecraft landing on an asteroid, and two spacecraft docking in low Earth orbit.

\end{abstract}


\section{Introduction}


\IEEEPARstart{A}{} common requirement for spacecraft systems is the capability for one spacecraft, called the \emph{chaser}, to dock with another spacecraft, called the \emph{target}. 
Successful docking requires the satisfaction of several tolerances,
such as a minimum and maximum relative velocity, maximum displacement between docking mechanisms, and maximum spacecraft attitude deviation, among others. 
In this paper, we propose encoding these tolerances as Control Barrier Functions (CBFs) \cite{CBF_Tutorial} and applying CBF theory to guarantee tolerance satisfaction. 

There are two principal approaches to robustness with CBFs. First, a CBF can be designed to drive state trajectories that lie outside the allowable \emph{safe set} into this set \cite{Robust_CBFs,tunable_safety}. Second, a CBF can be designed to ensure that a disturbance with a known upper bound never causes the state to leave the safe set \cite{RobustCBFsOld,garg2020robust,Automatica}. In this paper, we use the latter approach so that docking tolerances are never violated. 
However, tight-tolerance objectives, such as docking, in principle require that the system operate very close to the boundary of its safe set (e.g. the chaser comes very close to the target).
In the presence of a disturbance pointing toward the interior of the safe set, or no disturbance, a robust CBF may prevent the system from approaching sufficiently close to the boundary of its safe set to execute its mission. Thus, this paper develops conditions 
under which one can
guarantee that system trajectories always remain safe, yet also approach to within the required proximity to the boundary of the safe set in finite time. While our focus is 
spacecraft docking to other spacecraft, or landing on celestial bodies, the 
developed approach
can be applied to 
operating any system near the boundary of a safe set using CBFs.

Autonomous spacecraft rendezvous and docking has been extensively studied, and addressed using several methods, including artificial potential fields (APFs) \cite{Docking_APFs,Docking_Safety_APFs,Coupled_APF_journal,Adaptive_APF}, path planning \cite{OC}, model predictive control \cite{Docking_MPC}, sliding mode control \cite{SMC}, reinforcement learning \cite{RL}, and linear control \cite{adaptive}, among others. 
While the fundamental problem almost always centers on the Hill-Clohessy-Wiltshire (HCW) dynamics, different authors have considered various constraints. The work in \cite{Coupled_APF_journal,SMC} specifically considers APFs for coupled rotation and translation, while \cite{Adaptive_APF,OC} consider fuel efficiency as well. The work in \cite{adaptive} considers adaptation to uncertain model parameters, and \cite{OC} considers a tumbling target.
However, most of the aforementioned works attempt to accomplish docking exactly, or simply report the achieved tolerances when disturbances are added, rather than provably guaranteeing satisfaction of docking tolerances. Such guarantees can be obtained systematically through
CBFs, which can be used in conjunction with all of the above methods and constraints.

The author's prior work in \cite{Automatica} unifies the topics of input constraint satisfaction and disturbance rejection in CBFs applied to spacecraft problems. This paper builds upon the work in \cite{Automatica} by considering the case when the spacecraft mission and safety requirements are opposite to each other, as in the case of docking, and thus require operations within tight tolerances. The rest of this paper is organized as follows. Section~\ref{sec:preliminaries} presents notations and precise definitions of \emph{landing} and \emph{docking}. It then presents two important lemmas from \cite{Automatica} on guaranteeing input constraint satisfaction and disturbance rejection when employing CBFs, which this paper will build upon. 
Section~\ref{sec:methods} presents methods of accomplishing landing and docking within prescribed tolerances. Section~\ref{sec:simulations} presents simulations for a spacecraft landing on an asteroid with nontrivial gravity, and docking with another spacecraft in low Earth orbit. Section~\ref{sec:conclusions} presents concluding remarks.

\section{Preliminaries} \label{sec:preliminaries}

\noindent\textbf{Notations: }
Following the conventions in \cite{Automatica}, given a function $f(t,x)$ of time $t$ and state $x$, let $\partial_t f(t,x)$ denote the derivative in time $t$, and let $\nabla f(t,x)$ denote the gradient with respect to the vector $x$. Let $\dot{f}$ denote the total derivative $\dot{f} = \partial_t f + \nabla f \dot{x}$ and $\ddot{f}$ the second total derivative. Let $\|\cdot\|$ denote the 2-norm, and $\|\cdot\|_\infty$ the $\infty$-norm. Given a function $\kappa:\reals\rightarrow\reals$, let $\kappa^{-1}:\reals\rightarrow\reals$ denote the inverse function of $\kappa$ (if it exists). A function $\alpha:\reals_{\geq 0}\rightarrow\reals_{\geq 0}$ is said to belong to class-$\mathcal{K}$, denoted as $\alpha\in\mathcal{K}$, if it is strictly increasing and $\alpha(0)=0$. Define the function $\abssq:\reals\rightarrow\reals$ as $\abssq(\lambda) = \lambda | \lambda |$, and note that $\abssq$ is continuously differentiable and invertible everywhere on $\reals$.

\vspace{2pt}
\noindent\textbf{Problem Formulation: }
We consider a system of the form
\begin{equation}
    \dot{x} = f(t,x) + g(t,x)(u + w_u) + w_x \,, \label{eq:model}
\end{equation}
with time $t\in\mathcal{D} \triangleq [t_0, t_f]$, state $x\in\reals^n$, control input $u\in\mathcal{U}\subset\reals^m$ where $\mathcal{U}$ is compact, and continuous disturbances $w_u \in \reals^m$ and $w_x\in\reals^n$. Let $f:\mathcal{D}\times\reals^n\rightarrow\reals^n$ and $g:\mathcal{D}\times\reals^n\rightarrow\reals^{n\times m}$ be locally Lipschitz continuous functions, and assume there are known constants $w_{u,\textrm{max}}, w_{x,\textrm{max}} \in \reals$ such that $\|w_u\|\leq w_{u,\textrm{max}}$ and $\|w_x\|\leq w_{x,\textrm{max}}$. For the spacecraft docking problem, the matched disturbance $w_u$ could represent unmodelled forces such as drag and higher-order gravity, while the unmatched disturbance $w_x$ could represent filtered sensor updates when $x$ is an estimate of the true state.

Let the function $h:\mathcal{D}\times\reals^n\rightarrow\reals$ be a metric for distance between the chaser and target agents, defined so that $h < 0$ when the agents are separated, and $h=0$ when the agents are in contact. When $\dot{h}>0$, the chaser is approaching the target.
We define landing and docking as follows.

\begin{definition}[Landing]\label{def:landing}
    The state $x$ is said to correspond to \emph{landing} at time $t$ if $h(t,x) = 0$ and $\dot{h}(t,x) \in [0, \gamma_2]$ simultaneously, where $\gamma_2 \geq 0$ is a specified constant.
\end{definition}

\begin{definition}[Docking]\label{def:docking}
    The state $x$ is said to correspond to \emph{docking} at time $t$ if $h(t,x) = 0$ and $\dot{h}(t,x) \in [\gamma_1, \gamma_2]$ simultaneously, where $\gamma_1, \gamma_2 > 0$ are specified constants.
\end{definition}

Note that Definitions~\ref{def:landing}-\ref{def:docking} pose the landing/docking objective as the value of $\dot{h}$ belonging to an interval, rather than requiring a single value of $\dot{h}$.
Landing and docking differ only by the requirement of a minimum velocity $\gamma_1 > 0$ for docking. In both cases, we assume $\gamma_2$ is sufficiently small to be dissipated by the spacecraft structure; i.e. landing/docking is a \emph{controlled collision}. The focus of this paper is on ensuring that landing/docking occurs within the specified range of $\dot{h}$ values in finite time.

\vspace{2pt}
\noindent
\textbf{Background: }
We will accomplish this objective in part by treating the $\dot{h}$ upper bound requirement as a safety constraint, and then employing CBF techniques to ensure safety for all time. Let $H:\mathcal{D}\times\reals^n\rightarrow\reals$ encode one such safety constraint, where $H(t,x)\leq 0$ implies that the state $x$ is in the safe set at time $t$. We say $H$ is a CBF if it meets the following criteria.

\begin{definition}[{\cite[Def. 3]{Automatica}}]\label{def:cbf}
    For the system \eqref{eq:model}, a continuously differentiable function $H: \mathcal{D}\times \reals^n \rightarrow \reals$ is a \emph{control barrier function (CBF) on a set $\mathcal{X}$} if there exists a locally Lipschitz continuous $\alpha_0\in\mathcal{K}$ such that $\forall x\in\mathcal{X}(t), t\in\mathcal{D}$
    \begin{equation}
        \max_{\substack{\|w_u\|\leq w_{u,\textrm{max}} \\ \|w_x\| \leq w_{x,\textrm{max}}}} \inf_{u\in\mathcal{U}} \dot{H}(t,x,u,w_u,w_x) \leq \alpha_0(-H(t,x)) \,. \label{eq:cbf_definition}
    \end{equation}
\end{definition}

That is, $H$ is a CBF on $\mathcal{X}$ (where $\mathcal{X}$ can be any specified set for the purposes of Definition~\ref{def:cbf}) if there exists an allowable control input $u\in\mathcal{U}$ such that $\dot{H}$ is less than $\alpha_0$ even in the presence of a worst-case disturbance. 
For compactness, define
\begin{equation}
    \hspace{-2pt}W(t,x) \triangleq \| \nabla H(t,x) g(t,x)\| w_{u,\textrm{max}} + \| \nabla H(t,x) \| w_{x,\textrm{max}}, \hspace{-3pt} \label{eq:def_W}
\end{equation}
which represents the maximum contribution of the disturbances to 
$\dot{H}$. That is, $\dot{H}(t,x,u,w_u,w_x) \in [\dot{H}(t,x,u,0,0) - W(t,x),\, \dot{H}(t,x,u,0,0) + W(t,x)]$.
We can then establish safety with respect to the constraint $H$ as follows.

\begin{lemma}[{\cite[Cor. 17]{Automatica}}]\label{prior:cbf}
    Suppose $H:\mathcal{D}\times\reals^n\rightarrow\reals$ is a CBF on the set $\mathcal{S}(t) \triangleq \{x\in\reals^n\mid H(t,x) \leq 0\}$ 
    for the system \eqref{eq:model}. Suppose there exists constants $\eta_1, \eta_2 > 0$ such that $W$ in \eqref{eq:def_W} satisfies $W(t,x) \in [\eta_1,\eta_2], \forall x\in\mathcal S(t), t\in\mathcal{D}$. Let $\alpha_w\in\mathcal{K}$ be locally Lipschitz continuous. Then any control law $u(t,x)$ that is piecewise continuous in $t$ and locally Lipschitz continuous in $x$, and that satisfies: $\forall x\in\mathcal{S}(t), t\in\mathcal{D}$,
    \begin{equation}
       \hspace{-5pt} \partial_t H(\cdot) + \nabla H(\cdot) (f(\cdot) + g(\cdot) u(\cdot)) \hspace{-1pt} \leq \hspace{-1pt} \alpha_w(-H(\cdot)) W(\cdot) - W(\cdot) \hspace{-4pt} \label{eq:cbf_condition}
    \end{equation}
    where $(\cdot) = (t,x)$, will render the set $\mathcal{S}$ forward invariant.
\end{lemma}

\noindent
Note that by Definition~\ref{def:cbf}, if $H$ is a CBF on $\mathcal S$ and $W$ is bounded by $\eta_1$, $\eta_2$ as in Lemma~\ref{prior:cbf}, then there is at least one $\alpha_w\in\mathcal{K}$ for which there always exists a $u\in\mathcal{U}$ satisfying \eqref{eq:cbf_condition} for all $x\in\mathcal{S}(t), t\in\mathcal{D}$.
The set $\mathcal{S}$ is called the \emph{safe set}.

In this paper, we assume the function $h$ is of relative-degree 2 with respect to \eqref{eq:model}, and therefore not a CBF, so we construct a CBF from $h$ as follows. 
Since $h$ is of relative-degree 2, $\dot{h}$ is independent of $u$ and $w_u$, so define \cite[Eq. 13]{Automatica}
\begin{equation}
    \dot{h}_w(t,x) \triangleq \max_{\|w_x\| \leq w_{x,\textrm{max}} } \dot{h}(t,x,w_x) \,,
\end{equation}
which captures the worst-case value of $\dot{h}$ due to the unmatched disturbance. Assume $\|\nabla h\|$ does not vanish, so that $\dot{h}_w$ is differentiable. Recall the following result from \cite{Automatica}. 
\begin{lemma}[{\cite[Thm. 9]{Automatica}}]\label{prior:variable}
    Suppose $h$ is of relative-degree 2 with respect to \eqref{eq:model} and that there exists an invertible, monotone decreasing, and continuously-differentiable function $\Phi:\reals\rightarrow\reals$, whose derivative is $\phi:\reals\rightarrow\reals$, such that
    \begin{equation}
        \max_{\substack{\|w_u\|\leq w_{u,\textrm{max}} \\ \|w_x\| \leq w_{x,\textrm{max}}}} \inf_{u\in\mathcal{U}} \ddot{h}_w(t,x,u,w_u,w_x) \leq \phi(h(t,x)) \leq 0 \label{eq:variable_cbf_condition}
    \end{equation}
    for all $(t,x)$ such that $h(t,x) \leq 0$. Let $\Phi^{-1}:\reals\rightarrow\reals$ be the inverse function of $\Phi$. Then the function
    \begin{equation}
        H_0(t,x) \triangleq \Phi^{-1} \left( \Phi (h(t,x)) - \frac{1}{2} \abssq\left(\dot{h}_w(t,x)\right) \right) \label{eq:variable_cbf}
    \end{equation}
    is a CBF on the set $\mathcal{S}_0(t) \triangleq \{ x \in \reals^n \mid H_0(t,x) \leq 0 \textrm{ and } h(t,x) \leq 0\}$. Moreover, condition \eqref{eq:cbf_definition} is satisfied for any $\alpha_0\in\mathcal{K}$ on $\mathcal{S}_0$, and any controller satisfying the conditions of Lemma~\ref{prior:cbf} on $\mathcal{S}_0$ will render $\mathcal{S}_0$ forward invariant.
\end{lemma}

The distinction between the sets $\mathcal S$ in Lemma~\ref{prior:cbf} with $H=H_0$ and $\mathcal S_0$ in Lemma~\ref{prior:variable} is a technicality that arises because the set $\mathcal{S}\setminus\mathcal{S}_0$ where $H_0 \leq 0$ and $h>0$ is unreachable from $\mathcal{S}_0$; see \cite[Lemma 6]{Automatica} for more information. Physically, the function $\Phi$ in Lemma~\ref{prior:variable} represents a potential field, for example gravitational potential of a spacecraft around a gravitational source, and $H_0$ is analogous to potential energy, where the agent must maintain a specified minimum potential energy to remain ``safe'', i.e. to remain in $\mathcal{S}_0$. For our purposes, Lemma~\ref{prior:variable} is used to 1) convert the metric $h$ into a valid CBF under $\mathcal{U}$ that we can then use for control design, and 2) define the set of allowable initial conditions $x(t_0) \in \mathcal{S}_0(t_0)$.

\section{Methods} \label{sec:methods}

We divide the landing/docking problem into two parts: \emph{robust safety} in Section~\ref{subsec:safety} and \emph{robust proximity} in Section~\ref{subsec:convergence}. Robust safety refers to the requirement that, under any allowable disturbances $w_u,w_x$ in \eqref{eq:model}, $\dot{h}(t,x(t))\leq \gamma_2$ for all $t$ such that $h(t,x(t))= 0$. 
Robust proximity refers to the requirements that 1) $h(t,x(t)) = 0$ for finite $t$, and 2) $\dot{h}(t,x(t)) \geq \gamma_1$ for docking, or $\dot{h}(t,x(t)) \geq 0$ for landing, at the time when $h(t,x(t)) = 0$.

\subsubsection{Robust Safety} \label{subsec:safety}

The set $\mathcal{S}_0$ in Lemma~\ref{prior:variable} does not contain any states such that $h = 0$ and $\dot{h} > 0$ simultaneously (as such states would immediately leave $\mathcal{S}_0$), so for the same function $\Phi$ as in \eqref{eq:variable_cbf_condition}-\eqref{eq:variable_cbf} define the new function
\begin{equation}
    \hspace{-2pt} H_1(t,x) \triangleq \Phi^{-1} \left( \Phi(h(t,x)) - \frac{1}{2}\abssq\left(\dot{h}_w(t,x)\right) + \delta \right) \hspace{-2pt} \label{eq:new_cbf}
\end{equation}
for some parameter $\delta \geq 0$. This expands the set $\mathcal{S}_0$ to the set $\mathcal{S}_1(t) \triangleq \{ x\in\reals^n \mid H_1(t,x) \leq 0 \textrm{ and } h(t,x)\leq d\triangleq \Phi^{-1}(\Phi(0) - \delta)\}$, where $d \geq 0$. Unlike $\mathcal{S}_0$, the set $\mathcal{S}_1$ contains docking states. First, we note that $H_1$ is also a CBF.

\begin{theorem}\label{thm:new_cbf}
    Suppose the conditions of Lemma~\ref{prior:variable} hold for all $(t,x)$ such that $h(t,x)\leq d$. Then $H_1:\mathcal{D}\times\reals^n\rightarrow\reals$ in \eqref{eq:new_cbf} is a CBF on the set $\mathcal{S}_1$. Moreover, condition \eqref{eq:cbf_definition} is satisfied for any $\alpha_0\in\mathcal{K}$ on $\mathcal{S}_1$, and any controller satisfying the conditions of Lemma~\ref{prior:cbf} on $\mathcal{S}_1$ will render $\mathcal{S}_1$ forward invariant.
\end{theorem}

The proof of Theorem~\ref{thm:new_cbf} follows the same logic as \cite[Thm. 9]{Automatica} and is omitted for brevity. More importantly, the set $\mathcal{S}_1$ allows us to upper bound $\dot{h}$ when $h = 0$ as follows.

\begin{theorem} \label{thm:safety}
    If a trajectory $x(t),t\in\mathcal{D}$ satisfying $H_1(t,x(t))\leq 0, \forall t\in\mathcal{D}$ contains a point $(t_f,x(t_f))$ such that $h(t_f,x(t_f)) = 0$, then $\dot{h}(t_f,x(t_f)) \leq \sqrt{2\delta}$.
\end{theorem}
\begin{proof}
    The proof follows from the construction of $H_1$. Suppose there exists $t_f$ such that $h(t_f,x(t_f)) = 0$ and $H_1(t_f,x(t_f)) \leq 0$. Then
    \begin{align*}
        0 &\geq H_1(\cdot) = \Phi^{-1}\left( \Phi(h(\cdot)) - \frac{1}{2}\abssq \left(\dot{h}_w(\cdot)\right) + \delta \right), \\
        \Phi(0) &\leq \Phi\Big(\underbrace{h(t_f,x(t_f))}_{=0}\Big) - \frac{1}{2}\abssq\left(\dot{h}_w(t_f,x(t_f))\right) + \delta, \\
        0 &\leq -\frac{1}{2} \dot{h}_w(t_f, x(t_f)) | \dot{h}_w(t_f,x(t_f))| + \delta.
    \end{align*}
    Thus, $\sqrt{2\delta} \geq \dot{h}_w(t_f,x(t_f)) \geq \dot{h}(t_f,x(t_f))$.
\end{proof}

We then let $\delta = \frac{1}{2}\gamma_2^2$ and then any controller which renders $\mathcal{S}_1$ forward invariant will also ensure the landing/docking velocity meets the upper bound requirement in Definitions~\ref{def:landing}-\ref{def:docking}.

\subsubsection{Robust Proximity} \label{subsec:convergence}

Given a robustly safe control input from \eqref{eq:cbf_condition} and \eqref{eq:new_cbf}, the next problem is that of ensuring the trajectory reaches a landing/docking state, denoted $x_f$, under any allowable disturbances $w_u,w_x$. Note that the state $x_f$ is not necessarily an equilibrium of the system \eqref{eq:model}, so we do not require convergence to $x_f$. Rather, we require that the trajectory passes through an $x_f=x(t_f)$ meeting our criteria. Note that this paper does not consider the system evolution after the first time instance $t_f$ when landing/docking is achieved. 

Because of the disturbances, we cannot guarantee convergence of $h, \dot{h}, H$ to specific values. However, we can guarantee bounds on $H$, and by consequence $h$ and $\dot{h}$ as well.
To capture the possible impacts of the disturbances, define the set
\begin{equation}
    \partial_\epsilon \mathcal{S} \triangleq \{ x\in\mathcal{S} \mid H(t,x) \geq -\epsilon \} \label{eq:partial_set}
\end{equation}
where $\epsilon$ is a parameter. Define $\partial_\epsilon\mathcal{S}_0$ and $\partial_\epsilon\mathcal{S}_1$ similarly. 
While Lemma~\ref{prior:cbf} upper bounds $H$ 
(and by consequence upper bounds $\dot{h}$ in Theorem~\ref{thm:safety}),
the following result allows us to also lower bound $H$ for any allowable disturbances.

\begin{lemma}\label{lemma:approach}
    Suppose $H:\mathcal{D}\times\reals^n$ satisfies the assumptions of Lemma~\ref{prior:cbf}. If the control input $u(t,x)$ satisfies \eqref{eq:cbf_condition} with equality for all $x\in\mathcal{S}(t),t\in\mathcal{D}$, for some $\alpha_w\in\mathcal{K}$ in \eqref{eq:cbf_condition}, and the initial condition satisfies $x(t_0) \in \mathcal{S}(t_0)$, then $\lim_{t \rightarrow\infty} x(t) \in \partial_{\alpha_w^{-1}(2)} \mathcal{S}$ (i.e. $\partial_\epsilon\mathcal{S}$ as in \eqref{eq:partial_set} with $\epsilon=\alpha_w^{-1}(2)$).
\end{lemma}
\begin{proof}
    First, since
    $u(t,x)$ satisfies \eqref{eq:cbf_condition}, Lemma~\ref{prior:cbf} implies that closed-loop trajectories cannot leave $\mathcal{S}$ and thus $\limsup_{t\rightarrow\infty}H(t,x(t))\leq 0$. Moreover, the derivative of $H$ in the presence of disturbances is lower bounded by
    \begin{align}
        \dot{H} &= \partial_t H + \nabla H(f + g(u+w_u) + w_x) \nonumber \\
        &= \partial_t H + \nabla H(f + g u) + W + \nabla H g w_u + \nabla H w_x - W \nonumber \\
        &\overset{\eqref{eq:cbf_condition}}{=} \alpha_w(-H)W + \nabla H g w_u + \nabla H w_x - W \nonumber \\
        &\overset{\eqref{eq:def_W}}{\geq} \alpha_w(-H)W - 2W \label{eq:safety_proof}
    \end{align}
    where we omit the arguments $t,x$ for brevity. Since $W \geq \eta_1 > 0$ in Lemma~\ref{prior:cbf}, \eqref{eq:safety_proof} implies that $\dot{H}$ is strictly positive whenever $\alpha_w(-H) > 2$, or equivalently when $H < -\alpha_w^{-1}(2)$. It immediately follows that $\liminf_{t\rightarrow\infty} H(t,x(t)) \geq -\alpha_w^{-1}(2)$. Thus, as $t\rightarrow\infty$, the state $x(t)$ approaches $\partial_{\alpha_w^{-1}(2)} \mathcal{S}$.
\end{proof}

Thus, regardless of the disturbance, a control law satisfying \eqref{eq:cbf_condition} with equality guarantees a lower bound, determined by $\alpha_w$, on $H$ as $t\rightarrow\infty$. 
Moreover, because Theorem~\ref{thm:new_cbf} states that \eqref{eq:cbf_definition} holds for any $\alpha_0\in\mathcal{K}$, the choice of $\alpha_w$ is a free parameter.

\begin{remark} \label{rem:finite}
    Note that Lemma~\ref{lemma:approach} only guarantees $x\rightarrow \partial_{\alpha_w^{-1}(2)}\mathcal{S}$ as $t\rightarrow \infty$. There exist finite and fixed time extensions of Lemma~\ref{lemma:approach}, provided the derivative $\alpha_w'$ of $\alpha_w$ satisfies $\alpha_w'(\lambda_c) = \infty$ where $\lambda_c=\alpha_w^{-1}(2)$, such as $\alpha_w(\lambda) = \frac{2}{\lambda_c^{1/3}}((\lambda - \lambda_c)^{1/3} + \lambda_c^{1/3})$. However, this violates the assumption in Lemma~\ref{prior:cbf} that $\alpha_w$ is locally Lipschitz continuous. Instead, we employ the fact that for every $\epsilon > \alpha_w^{-1}(2)$, there exists a finite time $T$ such that $x\rightarrow \partial_{\epsilon}\mathcal{S}$ as $t\rightarrow T$.
\end{remark}

Note that Lemma~\ref{lemma:approach} applies to any CBF satisfying the assumptions of Lemma~\ref{prior:cbf}, and thus to $H_0$, $H_1$ as well. The next step is then to use the terminal set $\partial_{\alpha^{-1}_w(2)}\mathcal{S}$ in Lemma~\ref{lemma:approach} to generate the desired lower bounds on $h$ and $\dot{h}$. 
First, we introduce one more metric as follows.

\begin{definition}[Feasibility Margin]
    Let $l_h$ be the Lipschitz constant of $h$ in a neighborhood $\mathcal{Y}\subseteq\reals^n$ of the set $\{x\in\reals^n \mid h(t,x) = 0\}$. Given constants $\gamma_1, \gamma_2$, the \emph{feasibility margin} is
    \begin{equation}
        c_f(\gamma_1, \gamma_2) \triangleq \gamma_2 - (\gamma_1 + 2 l_h w_{x,\textrm{max}}) \label{eq:feasibility_margin} \,.
    \end{equation}
\end{definition}

\noindent
The feasibility margin is important because the condition $\dot{h}\in[\gamma_1, \gamma_2]$ is robustly guaranteed whenever $\dot{h}_w\in [\gamma_2 - c_f, \gamma_2]$. Thus, we require that $[\gamma_2 - c_f, \gamma_2]$ be nonempty, 
or equivalently that $c_f \geq 0$. 
Combining landing/docking with CBFs then further necessitates that $c_f > 0$, as in the following theorem. 

\begin{theorem} \label{thm:landing}
    Given constants $\gamma_2 > \gamma_1 > 0$ satisfying $c_f(\gamma_1,\gamma_2) > 0$ in \eqref{eq:feasibility_margin} and a function $H_1:\mathcal{D}\times\reals^n$ of the form \eqref{eq:new_cbf} with $\delta = \frac{1}{2}\gamma_2^2$, suppose $H_1$ meets the assumptions of Lemma~\ref{prior:cbf}. Suppose $\alpha_w\in\mathcal{K}$ in \eqref{eq:cbf_condition} satisfies
    \begin{equation}
        \hspace{-3pt}\alpha_w^{-1}(2) = -\Phi^{-1}\left( \half \gamma_2^2 + \Phi(0) - \half\left( 2 l_h w_{x,\textrm{max}} + \gamma_1\right)^2 \right) . \hspace{-2pt} \label{eq:alpha_property}
    \end{equation}
    If the control input $u(t,x)$ satisfies \eqref{eq:cbf_condition} with equality for all $x\in\mathcal{S}_1(t),t\in\mathcal{D}$ and the initial condition satisfies $x(t_0)\in\mathcal{S}_1(t_0)$, then there exists a finite $t_f\geq t_0$ such that $x(t_f)$ corresponds to landing at $t_f$.
\end{theorem}
\begin{proof}
    First, note that $c_f(\gamma_1,\gamma_2) > 0$ implies $\alpha_w^{-1}(2) > 0$ in \eqref{eq:alpha_property} (if instead it held that $\alpha_w^{-1}(2) \leq 0$, then $\alpha_w \notin \mathcal{K}$, so the objective would be infeasible). 
    
    Next, we note an important property of points inside the set $\partial_{\alpha_w^{-1}(2)}\mathcal{S}_1$, visualized in gray in Fig.~\ref{fig:explanation}. Let $x_c\in\reals^n$ be any point inside $\partial_{\alpha_w^{-1}(2)}\mathcal{S}_1(t_c)$ for some $t_c\in\mathcal{D}$, and suppose that $h(t_c, x_c) = 0$. Then at $x_c$, it holds that
    \begin{align}
        &-\alpha_w^{-1}(2) \leq H_1(t_c, x_c) \nonumber \\
        &\;\;\;\;\;\;\;\;\;\overset{\eqref{eq:new_cbf}}{=} \Phi^{-1}\bigg( \Phi\Big(\underbrace{h(t_c,x_c)}_{=0}\Big) - \frac{1}{2}\abssq\left(\dot{h}_w(t_c,x_c)\right) + \frac{1}{2}\gamma_2^2 \bigg) \nonumber \\
        &
        \Phi(-\alpha_w^{-1}(2)) \geq \Phi(0) - \frac{1}{2}\dot{h}_w(t_c,x_c)|\dot{h}_w(t_c,x_c)| + \frac{1}{2}\gamma_2^2 \label{eq:phi0_sub} \\
        &
        \dot{h}_w(t_c,x_c)|\dot{h}_w(t_c,x_c)| \geq  2\Phi(0) + \gamma_2^2 - 2 \Phi(-\alpha_w^{-1}(2)) \nonumber \\
        &
        \;\;\;\;\;\;\;\;\;\;\;\;\;\;\;\;\;\;\;\;\;\;\;\;\;\;\;\;\;\;\;\;\, \overset{\eqref{eq:alpha_property}}{=} (2 l_h w_{x,\textrm{max}} + \gamma_1)^2 \label{eq:hdot_lower_bounded}
    \end{align}
    The right hand side of \eqref{eq:hdot_lower_bounded} is positive, so $\dot{h}_w(t_c,x_c)$ will be positive as well. Specifically, $\dot{h}_w(t_c,x_c) \geq 2 l_h w_{x,\textrm{max}}+\gamma_1$. By definition, $l_h$ satisfies $l_h \geq \| \nabla h(t_c, x_c) \|$, and $\dot{h}$ satisfies $\dot{h}(t,x,w_x) \geq \dot{h}_w(t,x) - 2\|\nabla h(t,x)\| w_{x,\textrm{max}}$. It follows that $\dot{h}(t_c,x_c,w_x) \geq \dot{h}_w(t_c,x_c) - 2 l_h w_{x,\textrm{max}} \geq \gamma_1 > 0$. 
    Visually, this means that $\dot{h}_w(t_c,x_c)$ always lies on the intersection of the gray region and the magenta line in Fig~\ref{fig:explanation}, which implies $\dot{h}(t_c,x_c,w_x)$ always lies on the magenta line (i.e. the docking states).
    Moreover, since 
    $\Phi$ is monotone decreasing, 
    $\Phi(\lambda) > \Phi(0)$ for all $\lambda < 0$. It follows from \eqref{eq:phi0_sub} that if $x\in\partial_{\alpha_w^{-1}(2)}\mathcal{S}_1(t)$ and $h(t,x) < 0$, then $\dot{h}(t,x,w_x) > \gamma_1$, as shown by how the gray region in Fig.~\ref{fig:explanation} occurs for larger $\dot{h}_w$ (and therefore larger $\dot{h}$) as $h$ decreases.
    
    Safety with respect to $\gamma_2$ is already guaranteed by Theorem~\ref{thm:safety} since $\delta=\frac{1}{2}\gamma_2^2$, so we will now use the above property and Lemma~\ref{lemma:approach} to guarantee proximity, i.e. that there exists $t_f < \infty$ such that $h(t_f,x(t_f)) = 0$. 
    We will divide this into two cases, depending on $x(t_0)$. 
    The assumption $x(t_0)\in\mathcal{S}_1(t_0)$ implies $h(t_0,x(t_0)) \leq 0$, so going forward we assume $h(t,x(t))\leq 0$, as otherwise landing already occurred.
    
    First, in the case that $x(t_0) \notin \partial_{\alpha_w^{-1}(2)}\mathcal{S}_1(t_0)$, then by Lemma~\ref{lemma:approach}, $H_1$ is initially increasing and will keep increasing at least until the state reaches $\partial_{\alpha_w^{-1}(2)}\mathcal{S}_1$. 
    If the state reaches $\partial_{\alpha_w^{-1}(2)}\mathcal{S}_1$ before landing, then see the second case. 
    Otherwise, since $x(t)$ is converging to a set where $h(t,x(t)) < 0 \implies \dot{h}(t,x(t),w_x) > \gamma_1$, it follows from Remark~\ref{rem:finite} that for every $\gamma \in (0, \gamma_1)$, there exists a finite time $T(\gamma)$ such that the trajectory $x(t), t\geq T$ is sufficiently close to $\partial_{\alpha_w^{-1}(2)}\mathcal{S}_1$ that $h(t,x(t)) < 0 \implies \dot{h}(t,x(t),w_x) > \gamma$. For example, in Fig.~\ref{fig:explanation}, all states on the top-most orange line (level set where $H_1=-1.4\alpha_w^{-1}(2)$) satisfy $\dot{h} > \gamma = 0.1$ as long as $h < 0$. 
    Since $\dot{h}(t,x(t))$ becomes lower bounded by $\gamma>0$ within finite time $T$, there must exist a finite $t_f$ at which $h(t_f,x(t_f))=0$.
    
    Second, in the case that $x(T)\in\partial_{\alpha_w^{-1}(2)}\mathcal{S}_1(T)$ for any $T \in \mathcal{D}$ before landing occurs, then Lemma~\ref{lemma:approach} implies that $x(t)$ will remain in $\partial_{\alpha_w^{-1}(2)}\mathcal{S}_1$ for all $t\geq T$ (as occurs for the trajectories plotted with solid lines in Fig.~\ref{fig:explanation}). It follows from \eqref{eq:hdot_lower_bounded} that $\dot{h}(t,x(t),w_x) > \gamma_1$ for all $t\geq T$ as long as $h(t,x(t)) < 0$. Thus, $h$ will keep increasing until $t_f$ such that $h(t_f,x(t_f))=0$, and at $t_f$ it holds that $\dot{h}(t_f,x(t_f)) \geq \gamma_1$. Thus, landing as in Definition~\ref{def:landing} is guaranteed for finite $t_f$ in both cases. 
\end{proof}

\begin{figure}
    \centering
    \includegraphics[width=0.98\columnwidth]{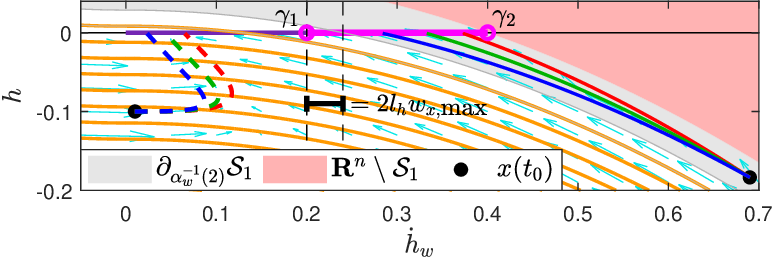}
    \caption{A phase diagram of $h$ and $\dot{h}_w$ when $w_x\equiv w_u \equiv 0$, along with trajectories from two initial states $x(t_0)$ inside and outside of $\partial_{\alpha_w^{-1}(2)}\mathcal{S}_1$. From each initial condition, three trajectories are propagated in the presence of: disturbances that decrease $\dot{H}_1$ (blue), no disturbance (green), and disturbances that increase $\dot{H}_1$ (red).
    The magenta line represents docking states. 
    The orange lines are level sets of $H_1$. Trajectories converge asymptotically to the $H_1 \geq -\alpha_w^{-1}(2)$ superlevel set (gray region), and therefore converge in finite time to the lower superlevel sets and to the horizontal axis.}
    \label{fig:explanation}
\end{figure}

That is, we have a condition, given by \eqref{eq:cbf_condition} and \eqref{eq:alpha_property},
under which landing is guaranteed. Note that while landing occurs for $\dot{h} \geq 0$, Theorem~\ref{thm:landing} requires us to encode the controller with a parameter $\gamma_1$ strictly greater than zero. Otherwise, landing is only guaranteed as $t\rightarrow\infty$. Also note that as the 
feasibility margin $c_f$ becomes smaller, the value $\lambda_c=\alpha_w^{-1}(2)$ in \eqref{eq:alpha_property} where $\alpha(\lambda_c) = 2$ becomes smaller, and thus
the slope of $\alpha_w$ becomes larger. In practice, for digital controllers, following a steep $\alpha_w$ curve will require a faster controller update cycle. Next, we cover the docking case as follows.

\begin{corollary} \label{cor:docking}
    Suppose the assumptions of Theorem~\ref{thm:landing}. If furthermore $x(t_0) \in \partial_{\alpha_w^{-1}(2)}\mathcal{S}_1(t_0)$, then there exists a finite $t_f\geq t_0$ such that $x(t_f)$ corresponds to docking at $t_f$.
\end{corollary}
\begin{proof}
    This result follows immediately from the proof of Theorem~\ref{thm:landing}. Since $x(t_0)\in\partial_{\alpha_w^{-1}(2)}\mathcal{S}_1(t_0)$, it follows that $x(t)$ remains in $\partial_{\alpha_w^{-1}(2)}\mathcal{S}_1$ for the entire closed-loop trajectory. By the argument in the final case of Theorem~\ref{thm:landing}, it follows that $\dot{h}(t,x(t)) \geq \gamma_1$ for all $t$ until $t_f$ such that $h(t_f,x(t_f))= 0$. Thus, docking as in Definition~\ref{def:docking} occurs for finite $t_f$.
\end{proof}

Thus, we can prescribe a minimum docking velocity $\gamma_1$ as well. 
Physically, the additional condition that $x(t_0) \in \partial_{\alpha_w^{-1}(2)}\mathcal{S}_1(t_0)$ means that the chaser agent begins the maneuver with a sufficiently large velocity relative to the target. If this is not the case, then the chaser agent may not be able to accelerate to the required velocity before contacting the target, as occurs for the trajectories plotted with dashed lines in Fig.~\ref{fig:explanation}.

\section{Simulations} \label{sec:simulations}

To verify the above conditions, we conducted two simulations. The first considered landing on an asteroid with nontrivial gravity and no atmosphere. The second considered docking in a low Earth orbit with multiple constraints.

For the first problem, we desire for a spacecraft to land on the surface of the asteroid Ceres. Let $r,v\in\reals^3$ be the position and velocity of the spacecraft with respect to the center of Ceres and $\mu = 6.26325(10)^{10} \textrm{ m}^3/\textrm{s}^2$ the gravitational parameter, so the dynamics are
\begin{equation}
    \dot{x} = \begin{bmatrix} \dot{r} \\ \dot{v} \end{bmatrix} = \begin{bmatrix} v \\ -\frac{\mu}{\|r\|^3} r \end{bmatrix} + \begin{bmatrix} 0 \\ u \end{bmatrix} + \begin{bmatrix} w_x \\ w_u \end{bmatrix} \,,
\end{equation}
where $u \in \mathcal{U} \subset \reals^3$ is the control input. Let $\mathcal{U} = \{ u \in \reals^3 \mid \|u\|_\infty \leq \bar{u}\}$ where $\bar{u} = 0.5 \textrm{ m/s}^2$, $w_{u,\textrm{max}} = 0.025 \textrm{ m/s}^2$, and $w_{x,\textrm{max}} = 0.01 \textrm{ m/s}$. Let $h$ be the distance from the surface of Ceres, modelled as a perfect sphere, $h = \rho - \|r\|$, where $\rho = 476000 \textrm{ m}$. 
We note that condition \eqref{eq:variable_cbf_condition} is satisfied for $\Phi(\lambda) = \frac{\mu}{\rho - \lambda} + (w_{u,\textrm{max}}-\bar{u})\lambda$  \cite{Automatica}. Let $H_1$ be as in \eqref{eq:new_cbf}, and choose $\gamma_1 = 0.1 \textrm{ m/s}$ and $\gamma_2 = 1.5 \textrm{ m/s}$, so $\delta = 1.125$ in \eqref{eq:new_cbf}. This places the zero vector outside $\mathcal{S}_1$, so $h$ and $H_1$ are differentiable everywhere on $\mathcal{S}_1$. For simplicity, choose $\alpha_w(\lambda) = k \lambda$, where 
condition \eqref{eq:alpha_property} implies
$k = 0.355$. We then apply the controller
\begin{equation}
    u(t,x) = \nabla H_1(t,x) g(t,x) \max_{\substack{b \nabla H_1 g \textrm{ satsifies \eqref{eq:cbf_condition}} \\ b \nabla H_1 g \in \mathcal{U} } } b  \label{eq:landing_law} \,,
\end{equation}
and simulated the spacecraft until landing occurred%
\footnotemark.
In practice, $u$ in \eqref{eq:landing_law} satisfies condition \eqref{eq:cbf_condition} with equality at all points except where \eqref{eq:cbf_condition} allows $H_1$ to increase at a rate that is unachievable within the input constraints. Since $H_1$ is constructed with $\Phi$ satisfying condition \eqref{eq:variable_cbf_condition}, the first constraint on the maximization in \eqref{eq:landing_law} will never require $H_1$ to decrease at a rate that is unachievable within the input constraints. Thus, the maximization in \eqref{eq:landing_law} is always feasible. For this simulation, we let $w_u,w_x$ be random bounded disturbances. The resultant trajectory is shown in Fig.~\ref{fig:landing_trajectory}, the altitude above Ceres is shown in Fig.~\ref{fig:landing_altitude}, and the control inputs are shown in Fig.~\ref{fig:landing_control}. As expected, the control inputs always remained within the allowable set $\mathcal{U}$, and the spacecraft achieved landing in 3236 seconds with $\dot{h}(t_f,x(t_f)) = 1.46 \textrm{ m/s} < \gamma_2$.

\begin{figure}
    \centering
    \includegraphics[width=0.9\columnwidth,trim={0in, 0in, 0.18in, 0.39in}, clip]{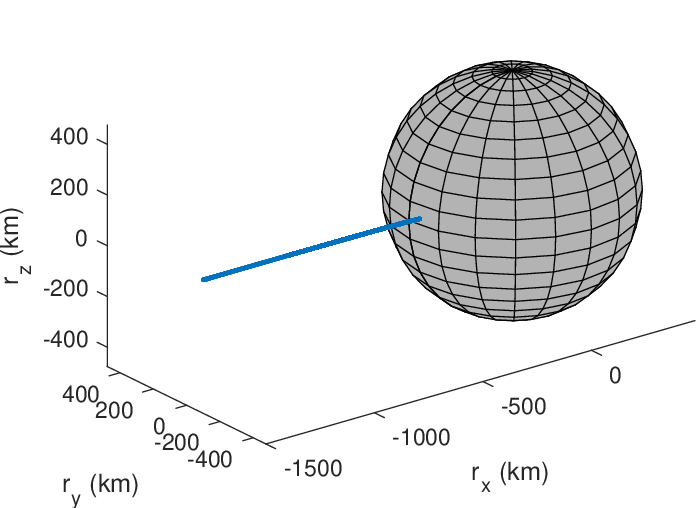}
    \caption{Trajectory of the spacecraft as it lands on Ceres}
    \label{fig:landing_trajectory} \vspace{12pt}
%
    \centering
    \includegraphics[width=\columnwidth]{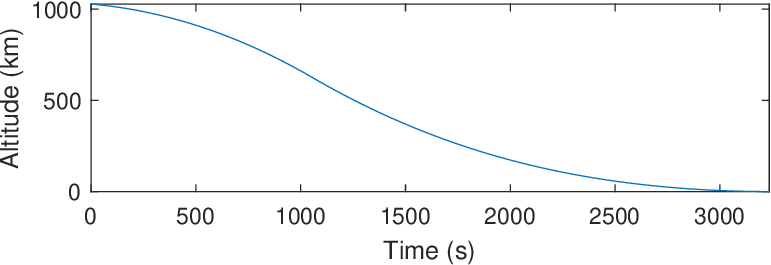}
    \caption{Altitude of the spacecraft as it lands on Ceres}
    \label{fig:landing_altitude} \vspace{12pt}
%
    \centering
    \includegraphics[width=\columnwidth]{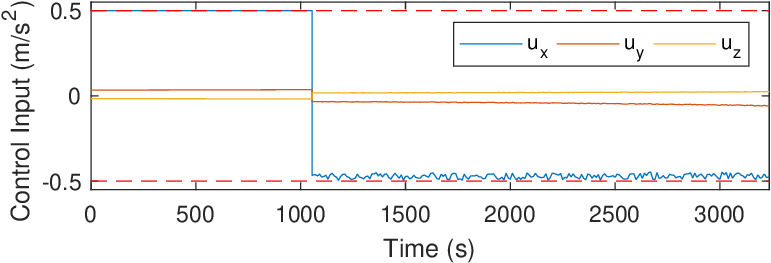}
    \caption{Control inputs of the spacecraft as it lands on Ceres}
    \label{fig:landing_control}
\end{figure}

For the second problem, we desire for a chaser spacecraft to dock with a target spacecraft in a 400 km altitude circular Earth orbit. Suppose that the chaser coordinates relative to the target are $x_1, x_2 \in \reals$ and follow the HCW dynamics
\begin{equation}
    \hspace{-2pt} \begin{bmatrix} \dot{x}_1 \\ \dot{x}_2 \\ \ddot{x}_1 \\ \ddot{x}_2 \end{bmatrix} = \begin{bmatrix} 0 & 0 & 1 & 0 \\ 0 & 0 & 0 & 1 \\ 3 n^2 & 0 & 0 & 2 n \\ 0 & 0 & -2 n & 0 \end{bmatrix} \begin{bmatrix} x_1 \\ x_2 \\ \dot{x}_1 \\ \dot{x}_2 \end{bmatrix} + \begin{bmatrix} 0 \\ 0 \\ u_1 \\ u_2 \end{bmatrix} + \begin{bmatrix} w_{x,1} \\ w_{x,2} \\ w_{u,1} \\ w_{u,2} \end{bmatrix} \hspace{-2pt} \label{eq:hcw}
\end{equation}
with $n = 0.00113 \textrm{ rad/s}$. Let $\mathcal{U} = \{ u\in\reals^2 \mid \| u\|_\infty \leq \bar{u}\}$ where $\bar{u} = 0.082 \textrm{ m/s}^2$, $w_{u,\textrm{max}} = 0.002 \textrm{ m/s}^2$, and $w_{x,\textrm{max}} = 0.001 \textrm{ m/s}$. For this problem, we desire that the chaser dock along a particular docking axis $\hat{a} = [0,\, 1]\transpose$, starting from behind the target (i.e. $x_2(t_0) < 0$). First, let $h = x_2$ encode the distance from the docking point, and let $H_1$ be a function of $h$ as in \eqref{eq:new_cbf} using $\Phi(\lambda) = -\tilde{u}_1 \lambda$ where $\tilde{u}_1 = 0.057 \textrm{ m/s}^2$. Let $\gamma_1 = 0.07 \textrm{ m/s}$ and $\gamma_2 = 0.12 \textrm{ m/s}$, so $\delta = 0.0072$. Again, let $\alpha_w(\lambda) = k_1\lambda$, where condition~\eqref{eq:alpha_property} implies $k_1 = 25$.
Next, to ensure convergence along the docking axis, define $h_r = x_1 - \Delta$ and $h_l = -x_{1} - \Delta$, where the tolerance $\Delta = 0.03 \textrm{ m}$.%
\footnotetext{All simulation code can be found at \url{https://github.com/jbreeden-um/phd-code/tree/main/2022}}
Then define the CBFs $H_{0,r}$ and $H_{0,l}$ as functions of $h_r$ and $h_l$, respectively, as in \eqref{eq:variable_cbf} using $\Phi(\lambda) = -\tilde{u}_0 \lambda$ where $\tilde{u}_0 = 0.021 \textrm{ m/s}^2$. For these two constraints, we use the original CBF $H_0$ in \eqref{eq:variable_cbf}, instead of the relaxed CBF $H_1$ in \eqref{eq:new_cbf} as we seek to guarantee that the tolerance $\Delta$ is not violated. For these two CBFs, apply the class-$\mathcal{K}$ functions $\alpha_r(\lambda) = \alpha_l(\lambda) = k_0 \lambda$ with $k_0 = 200$. Finally, we impose a velocity constraint $H_v(t,x) = \| [\dot{x}_1,\,\dot{x}_2] \|_\infty - v_{max}$ with $v_{max} = 10 \textrm{ m/s}$ and using $\alpha_v(\lambda) = k_v\lambda$ with $k_v = 20$. Define
\begin{multline}
    \hspace{-6pt}u_{nom}(t,x) = \big[ \alpha_w(-H_1(\cdot))W(\cdot) - W(\cdot) - \partial_t H_1(\cdot) \\ - \nabla H_1(\cdot) f(\cdot) \big] \nabla H_1(\cdot) g(\cdot) / \|\nabla H_1(\cdot) g(\cdot) \|^2 - k_p x_1 \hspace{-6pt} \label{eq:u_nom}
\end{multline}

\begin{figure}
    \centering
    \includegraphics[width=\columnwidth]{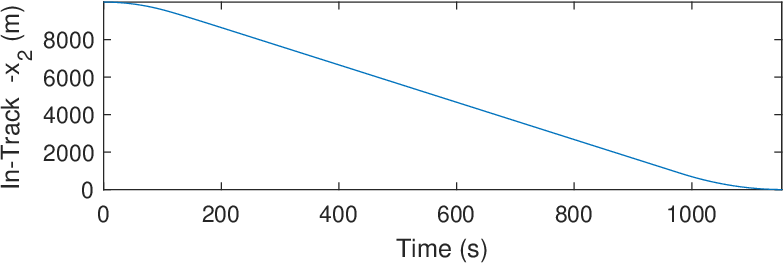}
    \caption{Distance between spacecraft along docking axis}
    \label{fig:docking_distance}
\end{figure}

\begin{figure}
    \centering
    \includegraphics[width=\columnwidth]{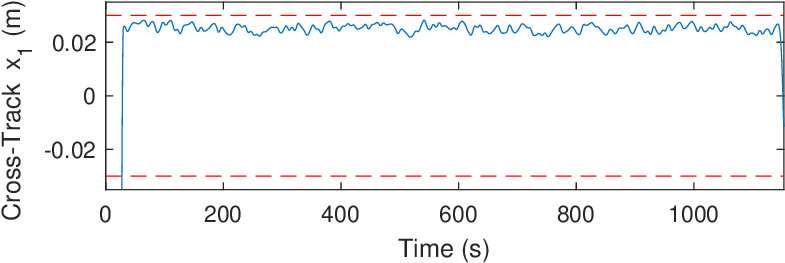}
    \caption{Distance between spacecraft orthogonal to docking axis}
    \label{fig:docking_cross}
\end{figure}

\begin{figure}[t!]
    \centering
    \includegraphics[width=\columnwidth]{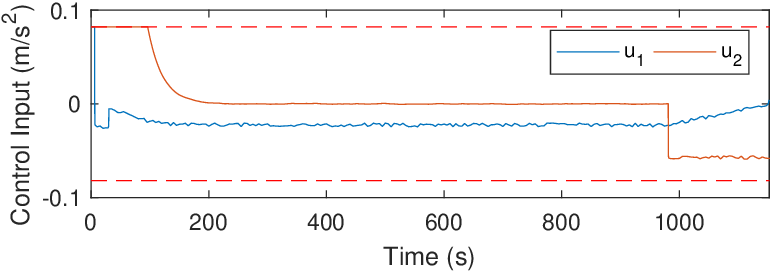}
    \caption{Control inputs of the chaser spacecraft}
    \label{fig:docking_control}
\end{figure}

\noindent
where $k_p = 0.1$. We then apply the controller
\begin{align}
    &u(t,x) = \label{eq:piecewise_controller} \\ &\;\;\;\;\;\;\;\;\;\; \begin{cases} \displaystyle \argmin_{\substack{u\in \mathcal{U}, \\ u \textrm{ satisfies \eqref{eq:cbf_condition} for } H_1, \\ u \textrm{ satisfies \eqref{eq:cbf_condition} for } H_{0,r} \\ u \textrm{ satisfies \eqref{eq:cbf_condition} for } H_v} } \hspace{-14pt} \| u - u_{nom}(t,x) \|^2 & H_{0,l}(t,x) > 0 \\
    \displaystyle \argmin_{\substack{u\in \mathcal{U}, \\ u \textrm{ satisfies \eqref{eq:cbf_condition} for } H_1, \\ u \textrm{ satisfies \eqref{eq:cbf_condition} for } H_{0,r}, \\ u \textrm{ satisfies \eqref{eq:cbf_condition} for } H_{0,l}  \\ u \textrm{ satisfies \eqref{eq:cbf_condition} for } H_v} } \hspace{-16pt} \| u - u_{nom}(t,x) \|^2 & H_{0,l}(t,x) \leq 0 \end{cases} \nonumber \,.
\end{align}
Note that the controller in \eqref{eq:piecewise_controller} is broken into two cases because for most initial conditions, only one of the CBFs $H_{0,r},H_{0,l}$ will be nonpositive, in this case $H_r(t_0,x(t_0)) \leq 0$. The term $-k_p x_1$ of $u_{nom}$ works to drive the spacecraft close to the docking axis, and once it is sufficiently close, we apply both the $H_{0,r}$ and $H_{0,l}$ safety conditions. Meanwhile, the first term of $u_{nom}$ in \eqref{eq:u_nom} satisfies condition \eqref{eq:cbf_condition} with equality to ensure docking occurs in finite time, and the quadratic program in \eqref{eq:piecewise_controller} ensures safety and input constraint satisfaction. Note that $\tilde{u}_1$ and $\tilde{u}_0$ were chosen so that \eqref{eq:variable_cbf_condition} is \emph{strictly} satisfied for each $\Phi$. This results in the QPs in \eqref{eq:piecewise_controller} being \emph{strictly} feasible and therefore locally Lipschitz continuous in $x$ \cite[Thm. 2.1]{slaters}.

A docking simulation
with random bounded disturbances is shown in Figs.~\ref{fig:docking_distance}-\ref{fig:docking_cross} and the control inputs are shown in Fig.~\ref{fig:docking_control}. Docking occurred within the constraints in 1153 seconds with a terminal velocity of $\dot{h}(t_f,x(t_f)) = 0.11 \textrm{ m/s} \in [\gamma_1,\gamma_2]$. We note that $x_1$ in Fig.~\ref{fig:docking_cross} converged quickly to $x_1 \in[-\Delta,\Delta]$, and then spent a lot of time near $x_1 = \Delta$. This is because the uncontrolled dynamics in \eqref{eq:hcw} tend to cause $x_1$ to increase. The proportional control law $-k_p x_1$ in \eqref{eq:u_nom} does not adequately account for these dynamics, so instead the safety constraint $H_{0,r}$ ensures that $x_1 \leq \Delta$ for all time.

\section{Conclusion} \label{sec:conclusions}

We have demonstrated how robust CBFs introduce a margin on how close system trajectories can come to the boundary of the safe set. We then developed a method for tuning this margin and applied it to guaranteeing the finite-time execution of a landing and docking maneuver with a terminal velocity inside a specified interval in the presence of bounded matched and unmatched disturbances. Future work includes studying appropriate CBFs for longer maneuvers with additional obstacles, as well as the application of these techniques under digital controllers with fixed update cycles.

\bibliographystyle{ieeetran}
\bibliography{sources}

\end{document}